\documentclass{amsart}
\usepackage{latexsym}
\def\qed{\hfill $\Box$}

\usepackage{amsfonts,amssymb,amsmath}
\usepackage{amsthm}
\usepackage{graphicx,color}
\newtheorem{theorem}{Theorem}	%'藝
\newtheorem{lemma}{Lemma}[section]		%•â'è
\newtheorem{corollary}{Corollary}	%Œn	
\newtheorem{proposition}{Proposition}			%–½'è
\newtheorem{definition}{Definition}
\newtheorem{example}{Example}[section]
\newtheorem{remark}{Remark}[section]	%'è‹`
\newfont{\bg}{cmr9 scaled\magstep4}
\newcommand{\bigzerol}{\smash{\lower1.0ex\hbox{\bg 0}}}

\title{%Classification of\\ 
Recognizable classification of Lorentzian distance-squared mappings
\\
}
\author{Shunsuke Ichiki}
\address{Graduate School of Environment and Information Sciences,  
Yokohama National University, 
Yokohama 240-8501, Japan}
\email{ichiki-shunsuke-jb@ynu.ac.jp}
\author{Takashi Nishimura}%\footnote{Corresponding author.   }}
\address{Research Group of Mathematical Sciences,  
Research Institute of Environment and Information Sciences,  
Yokohama National University, 
Yokohama 240-8501, Japan}
\email{nishimura-takashi-yx@ynu.jp}
\thanks{T.~Nishimura was partially supported 
by JSPS and CAPES under the Japan--Brazil research cooperative 
program} 
\begin{document}
\date{}
%%%%%%%%%%%%%%%%%%%%%%%%%%%%%%%%%%%%%%%%%%%%% 
\begin{abstract}
The Lorentzian length, 
which is one of the most significant functions in Lorentzian geometry, 
is a complex-valued function. 
Its square gives a real-valued non-degenerate quadratic function.   
%the application of singularity theory to differential geometry.
In this paper, we define naturally extended mappings of 
Lorentzian distance-squared functions, wherein each component 
is a Lorentzian distance-squared function; and classify these mappings completely 
by the likeness of recognition subspaces.     
\end{abstract}
\subjclass{57R45, 58C25, 58K50} 
\keywords{Lorentzian distance-squared mapping, 
distance-squared mapping, 
recognition subspace, 
definite fold mapping, 
Lorentzian indefinite fold mapping} 
%%%%%%%%%%%%%%%%%%%%%%%%%%%%%%%%%%%%%%%%%%%%%%  
\maketitle
\noindent
%% main text
%\begin{document}
\section{Introduction}\label{Introduction}
%%%%%%%%%%%%%%%%%%%%%%%%%%%%%%%%%%%%%%%%%%%%%%%%%% 
%%%%%%%%%%%%%%%%%%%%%%%%%%%%%%%%%%%%%%%%%%%%%%%%%% 
Let $n$ be a positive integer.   
For the $(n+1)$-dimensional vector space $\mathbb{R}^{n+1}$,  
the following quadratic form is called the {\it Lorentzian inner product}: 
\[ \langle x, y\rangle =-x_0y_0+x_1 y_1+\cdots +x_n y_n,\]
where $x=(x_0, x_1, \ldots, x_n), y=(y_0, y_1, \ldots, y_n)$ 
are elements of $\mathbb{R}^{n+1}$.         
The $(n+1)$-dimensional vector space $\mathbb{R}^{n+1}$ is called  
{\it Lorentzian $(n+1)$-space} and is denoted by $\mathbb{R}^{1,n}$ if 
the role of the Euclidean inner product $x\cdot y=\sum_{i=0}^nx_iy_i$ is replaced by 
the Lorentzian inner product.    
For a vector $x$ of 
Lorentzian $(n+1)$-space $\mathbb{R}^{1,n}$, $\sqrt{\langle x,x\rangle}$  is called  
the {\it Lorentzian length} of $x$.   
Note that 
the Lorentzian length %$\langle x,x\rangle$ 
may take a pure imaginary value and thus it does not give a real-valued function.   
On the other hand, 
its square gives  a real-valued 
quadratic function 
$x\mapsto \langle x, x\rangle$. 
% which may take negative values.    
A non-zero vector $x\in \mathbb{R}^{1,n}$ is said to be 
{\it space-like}, {\it light-like} or {\it time-like} if 
%$\langle x,x\rangle$ 
its Lorentzian length 
is positive, zero or pure imaginary respectively. 
The likeness of the vector subspace is defined as following (see Figure 1).       
\begin{definition}[\cite{ratcliffe}]\label{likeness}
{\rm 
Let $V$ be a vector subspace of $\mathbb{R}^{1,n}$.   
Then $V$ is said to be 
\begin{enumerate}
\item {\it time-like} if and only if $V$ has a time-like vector, 
\item {\it space-like} if and only if every nonzero vector in $V$ is space-like, or 
\item {\it light-like} otherwise.    
\end{enumerate}
}
\end{definition}
\begin{figure}
\begin{center}
\includegraphics[width=6cm,clip]{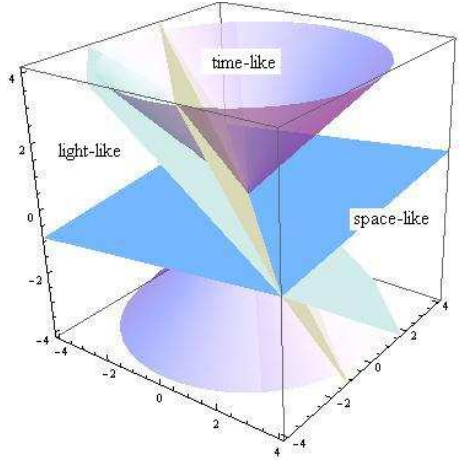}
\caption{Figure of Definition \ref{likeness}}
\end{center}
\end{figure}
\noindent 
The {\it light cone} of Lorentzian $(n+1)$-space $\mathbb{R}^{1,n}$, 
denoted by $LC$, is the set of  $x\in \mathbb{R}^{1,n}$ such that 
$\langle x, x\rangle =0$.   
%The time-axis $\{(x_0, 0, \ldots, 0)\; |\; x_0\in \mathbb{R}\}$ is denoted by 
%$T$.   
For more details on Lorentzian space, 
refer to \cite{ratcliffe}.    
Recently, geometry of submanifolds in Lorentzian space has been actively studied 
from the viewpoint of Singularity Theory 
(for instance, see \cite{izumiyakossowskipeicarmen,izumiyanunocarmen,izumiyapeisano,
izumiyasaji,izumiyasato,izumiyatakahashitari,izumiyatari1,izumiyatari2,izumiyahandan,kasedou,sato,tari1,tari2}).     
In this paper, 
we give a different application of Singularity Theory to the study of Lorentzian space 
from these researches.
\par 
For any point $p$ of $\mathbb{R}^{1,n}$, the following function  
$\ell_p^2: \mathbb{R}^{1,n}\to \mathbb{R}$ is called 
the {\it Lorentzian distance-squared function}(\cite{izumiyakossowskipeicarmen}):   
\[
\ell_p^2(x)=\langle x-p, x-p\rangle .   
\]
For finitely many points $p_0, \ldots, p_{k}\in \mathbb{R}^{1,n}$ $(1\le k)$, 
the {\it Lorentzian distance-squared mapping}, denoted by  
$L_{(p_0, \ldots, p_k)}: \mathbb{R}^{1,n}\to \mathbb{R}^{k+1}$, is defined as follows:   
\[
L_{(p_0, \ldots, p_k)}(x)=\left(\ell_{p_0}^2(x), \ldots, \ell_{p_k}^2(x)\right).   
\]
%In this paper, we investigate the fundamental properties of the Lorentzian 
%squared mappings.   
%Let $k$ be a positive integer satisfying $k \le n+1$.    
Define the vector subspace 
$V(p_0, \ldots, p_k)$ of $\mathbb{R}^{1,n}$, called the {\it recognition subspace}, 
by  
\[
V(p_0, \ldots, p_k)= 
\sum_{i=1}^k \mathbb{R}\;\overrightarrow{p_0 p_i}.    
\]
The $(k+1)$ points 
$p_0, \ldots, p_k$ are said to be 
{\it in general position} if 
the dimension of $V(p_0, \ldots, p_k)$ is $k$.   
For $(k+1)$ points $q_0, \ldots, q_k\in \mathbb{R}^{1,n}$ in general position $(k\le n)$, 
the singular set of $L_{(q_0, \ldots, q_k)}: \mathbb{R}^{1,n}\to \mathbb{R}^{k+1}$ is the $k$-dimensional 
affine subspace spanned by these points.    
%; and is said to be 
%{\it in Lorentzian general position} if 
%$\dim V(p_0, \ldots, p_k)=k$ and 
%$V(p_0, \ldots, p_k)$ does not contain the whole time-axis 
%$\mathbb{R}\times (\underbrace{0, \ldots, 0)}_{n-\mbox{tuples}}$).   
\par 
Two mappings $f, g: \mathbb{R}^{1,n}\to \mathbb{R}^{k+1}$ are said to be 
{\it $\mathcal{A}$-equivalent} if there exist $C^\infty$ diffeomorphisms 
$h: \mathbb{R}^{1,n}\to \mathbb{R}^{1,n}$ and 
$H: \mathbb{R}^{k+1}\to \mathbb{R}^{k+1}$ such that $g=H\circ f\circ h$.     
%and is said to be {\it Lorentzian $\mathcal{A}$-equivalent} 
%if there exist $C^\infty$ diffeomorphisms 
%$h: \mathbb{R}^{1,n}\to \mathbb{R}^{1,n}$ of the form 
%$h(x_0, x_1, \ldots, x_n)=(h_1(x_0), h_2(x_0, \ldots, x_n))$ where 
%$h_1: \mathbb{R}\to \mathbb{R}, h_2: \mathbb{R}^{1,n}\to \mathbb{R}^{n}$ 
%and 
%$H: \mathbb{R}^{k+1}\to \mathbb{R}^{k+1}$ such that $g=H\circ f\circ h$.   
For any two positive integers $k, n$ satisfying $k \le n$, 
the {\it normal form of definite fold mapping} is the following 
$\Phi_k: \mathbb{R}^{1,n}\to \mathbb{R}^{k+1}$: 
%mapping defined by 
\[
\Phi_k\left(x_0, x_1, \ldots, x_n\right)=
\left(x_1, \ldots, x_k, x_0^2+\sum_{i=k+1}^{n}x_i^2\right).
\]     
For any two positive integers $k, n$ satisfying $k < n$, 
the {\it normal form of Lorentzian indefinite fold mapping} is the following 
$\Psi_k: \mathbb{R}^{1,n}\to \mathbb{R}^{k+1}$: 
%the mapping defined by 
\[
\Psi_k\left(x_0, x_1, \ldots, x_n\right)=
\left(x_1, \ldots, x_k, 
-x_0^2+\sum_{i=k+1}^{n}x_i^2\right).
\]   
\begin{theorem}\label{theorem 1}
\begin{enumerate}
\item 
Let $k, n$ be two positive integers satisfying $k< n$ 
and let $p_0, \ldots, p_k$ belonging to 
$\mathbb{R}^{1,n}$ be $(k+1)$ points in general position. 
%Given $2$ integers $k, n$ satisfying $1\le k <n$ and 
%$(k+1)$ points $p_0, \ldots, p_k\in \mathbb{R}^{1,n}$ in general position, 
Then, the following hold:   
\begin{enumerate}
\item The Lorentzian distance-squared mapping 
$L_{(p_0, \ldots, p_k)}: \mathbb{R}^{1,n}\to \mathbb{R}^{k+1}$ is 
$\mathcal{A}$-equivalent to the normal form of 
definite fold mapping $\Phi_k$ if and only if 
the recognition subspace $V(p_0, \ldots, p_k)$ is time-like.   
% with  
%$V(p_0, \ldots, p_k)\cap T=\{0\}$, 
%where $T$ stands for $\{(x_0, 0, \ldots, 0)|\; x_0\in \mathbb{R}\}$.    
%\item The Lorentzian distance-squared mapping $L_{(p_0, \ldots, p_k)}$ is 
%$\mathcal{A}$-equivalent to the mapping 
%$(x_0, \ldots, x_n)\mapsto (x_0, x_1\ldots, x_{k-1}, \sum_{i=k}^{n}x_i^2)$ 
%if and only if 
%there exists a time-like vector in $V(p_0, \ldots, p_k)$ and 
%the recognition subspace 
%$V(p_0, \ldots, p_k)$ is time-like with 
%$V(p_0, \ldots, p_k)\cap T=T$. 
%the whole time-axis.    
%$\mathbb{R}\times (\underbrace{0, \ldots, 0)}_{n-\mbox{tuples}}$).   \item The Lorentzian norm squared mapping $L_{(p_0, \ldots, p_k)}$ is 
\item The Lorentzian distance-squared mapping 
$L_{(p_0, \ldots, p_k)}:\mathbb{R}^{1,n}\to \mathbb{R}^{k+1}$ is 
$\mathcal{A}$-equivalent to the normal form of 
Lorentzian indefinite fold mapping $\Psi_k$ if and only if 
the recognition subspace 
$V(p_0, \ldots, p_k)$ is space-like. 
\item The Lorentzian distance-squared mapping 
$L_{(p_0, \ldots, p_k)}:\mathbb{R}^{1,n}\to \mathbb{R}^{k+1}$ is 
$\mathcal{A}$-equivalent to the mapping 
\[
(x_0, \ldots, x_n)\mapsto \left( x_1, \ldots, x_k, x_0x_1
+\sum_{i=k+1}^n x_i^2\right)
\]
if and only if 
%the intersection $V(p_0, \ldots, p_k)\cap LC$ is a line.  
the recognition subspace 
$V(p_0, \ldots, p_k)$ is light-like. 
\end{enumerate}
\item 
Let $n$ be a positive integer and 
let $p_0, \ldots, p_n\in \mathbb{R}^{1,n}$ be $(n+1)$ points in general position. 
%Given $2$ integers $k, n$ satisfying $1\le k <n$ and 
%$(k+1)$ points $p_0, \ldots, p_k\in \mathbb{R}^{1,n}$ in general position, 
Then, the following hold:   
\begin{enumerate}
\item The Lorentzian distance-squared mapping 
$L_{(p_0, \ldots, p_n)}:\mathbb{R}^{1,n}\to \mathbb{R}^{n+1}$ is 
$\mathcal{A}$-equivalent to the normal form of 
definite fold mapping $\Phi_n$ if and only if 
%the intersection $V(p_0, \ldots, p_k)\cap LC$ is only the origin.  
the recognition subspace 
$V(p_0, \ldots, p_n)$ is time-like or space-like. 

\item The Lorentzian distance-squared mapping 
$L_{(p_0, \ldots, p_n)}: \mathbb{R}^{1,n}\to \mathbb{R}^{n+1}$ is 
$\mathcal{A}$-equivalent to the mapping  
$(x_0, \ldots, x_n)\mapsto (x_1, \ldots, x_n, x_0x_1)$ 
if and only if 
%the intersection $V(p_0, \ldots, p_k)\cap LC$ is a line.  
the recognition subspace 
$V(p_0, \ldots, p_n)$ is light-like.   
%space-like or light-like and 
%there exists a light-like vector in $V(p_0, \ldots, p_n)$.   
\end{enumerate}
\item 
Let $k, n$ be two positive integers satisfying $n<k$ and 
let $p_0, \ldots, p_{k}\in \mathbb{R}^{1,n}$ be $(k+1)$ points such that 
the $(n+2)$ points $p_0, \ldots, p_{n+1}$ are in general position. 
Then, the Lorentzian distance-squared mapping 
$L_{(p_0, \ldots, p_k)}: \mathbb{R}^{1,n}\to \mathbb{R}^{k+1}$ is always 
$\mathcal{A}$-equivalent to the inclusion 
$(x_0, \ldots, x_n)\mapsto (x_0, \ldots, x_n, 0, \ldots, 0)$.    
\end{enumerate}
\end{theorem}
\begin{example}\label{example 1.1}
{\rm Let $f, g, h : \mathbb{R}^{3}\to \mathbb{R}^2$ be polynomial mappings 
defined as follows:    
\begin{eqnarray*}
f(x, y, z) & = & \left(-x^2+y^2+z^2, -(x-1)^2+(y-1)^2+z^2\right),  \\  
g(x, y, z) & = & \left(-x^2+y^2+z^2, -(x-1)^2+(y-2)^2+z^2\right), \\ 
h(x, y, z) & = & \left(-x^2+y^2+z^2, -(x-2)^2+(y-1)^2+z^2\right).  
\end{eqnarray*} 
Since $\langle (1,1,0), (1,1,0)\rangle =0$, $\langle (1,2,0), (1,2,0)\rangle >0$ and 
$\langle (2,1,0), (2,1,0)\rangle <0$,  
%by Theorem \ref{theorem 1}, 
the mapping $f$ is $\mathcal{A}$-equivalent to the mapping 
$(x, y, z)\mapsto (y, xy+z^2)$, the mapping $g$ is $\mathcal{A}$-equivalent to 
the normal form of Lorentzian indefinite fold mapping $\Psi_1$ 
and the mapping $h$ is $\mathcal{A}$-equivalent to 
the normal form of definite fold mapping $\Phi_1$ 
respectively.}
\end{example}

Any non-singular fiber of $\Phi_{n-1}$ is a circle.     
Any non-singular fiber of $\Psi_{n-1}$ is an equilateral hyperbola.   
Any non-singular fiber of $(x_0, \ldots, x_n)\mapsto (x_1, \ldots, x_{n-1}, x_0x_1+x_n^2)$ is a parabola (possibly at infinity).    
Therefore, as a corollary of Theorem \ref{theorem 1}, we have the following:   
\begin{corollary}\label{corollary 1}
Let $n$ be a positive integer such that $\; 2\le n$ and 
let $p_0, \ldots, p_{n-1}$ belonging to 
$\mathbb{R}^{1,n}$ be $n$ points in general position. 
Then, the following hold:   
\begin{enumerate}
\item There exists a $C^\infty$ diffeomorphism $h:\mathbb{R}^{1,n}\to \mathbb{R}^{1,n}$ 
by which any non-singular fiber $L_{(p_0, \ldots, p_{n-1})}^{-1}(y)$ is mapped to a circle 
if and only if 
the recognition subspace $V(p_0, \ldots, p_{n-1})$ is time-like.   
\item There exists a $C^\infty$ diffeomorphism $h:\mathbb{R}^{1,n}\to \mathbb{R}^{1,n}$ 
by which any non-singular fiber $L_{(p_0, \ldots, p_{n-1})}^{-1}(y)$ is mapped to 
an equilateral hyperbola 
if and only if 
the recognition subspace $V(p_0, \ldots, p_{n-1})$ is space-like.  % with  
\item There exists a $C^\infty$ diffeomorphism $h:\mathbb{R}^{1,n}\to \mathbb{R}^{1,n}$ 
by which any non-singular fiber $L_{(p_0, \ldots, p_{n-1})}^{-1}(y)$ is mapped to a parabola 
if and only if 
the recognition subspace $V(p_0, \ldots, p_{n-1})$ is light-like.  
\end{enumerate}
\end{corollary}
\noindent 
It turns out that an affine diffeomorphism can be chosen 
as the diffeomorphism $h:\mathbb{R}^{1,n}\to \mathbb{R}^{1,n}$ in Corollary \ref{corollary 1} 
(see Remark \ref{remark 3.1}).   
\par 
\smallskip 
Our motivation to classify Lorentzian distance-squared mappings is the 
following Proposition \ref{proposition 1}. 
It is natural to ask how Proposition \ref{proposition 1} changes 
if distance-squared functions are replaced with Lorentzian 
distance-squared functions.
For any point $p$ of $\mathbb{R}^{n+1}$, the following function  
$d_p^2: \mathbb{R}^{n+1}\to \mathbb{R}$ is called 
the {\it distance-squared function}(\cite{brucegiblin}):   
\[
d_p^2(x)=(x-p)\cdot (x-p),    
\]
where the dot in the center stands for the Euclidean inner product of two vectors.   
For any finitely many points $p_0, \ldots, p_{k}\in \mathbb{R}^{n+1}$ $(1\le k)$, 
the following mapping $D_{(p_0, \ldots, p_k)}: \mathbb{R}^{n+1}\to \mathbb{R}^{k+1}$ has been introduced in \cite{ichikinishimura} 
and is called the {\it distance-squared mapping}:    
\[
D_{(p_0, \ldots, p_k)}(x)=\left(
d_{p_0}^2(x), \ldots, d_{p_k}^2(x)\right).   
\]
The following characterization has been known for distance-squared mappings.  
\begin{proposition}[\cite{ichikinishimura}]\label{proposition 1}
\begin{enumerate}
\item Let $k, n$ be two positive integers satisfying $k \le n$ and 
let $p_0, \ldots, p_k\in \mathbb{R}^{1,n}$ be $(k+1)$ points in general position.   
Then, the distance-squared mapping $D_{(p_0, \ldots, p_k)}$ is $\mathcal{A}$-equivalent 
to the normal form of definite fold mapping $\Phi_k$.   
\item 
Let $k, n$ be two positive integers satisfying $n<k$ and 
let $p_0, \ldots, p_{k}\in \mathbb{R}^{1,n}$ be $(k+1)$ points such that 
the $(n+2)$ points $p_0, \ldots, p_{n+1}$ are in general position.    
%Let $k, n$ be two positive integers satisfying $n<k$ and 
%let $p_0, \ldots, p_{k}\in \mathbb{R}^{1,n}$ be $(n+2)$ points such that 
%$p_0, \ldots, p_{n+1}$ are in general position.   
Then, the distance-squared mapping $D_{(p_0, \ldots, p_k)}$ is $\mathcal{A}$-equivalent 
to the inclusion $(x_0,\ldots,x_n)\mapsto (x_0,\ldots,x_n,0,\ldots,0)$.   
\end{enumerate}
\end{proposition}
%It is clear that neither the normal form of Lorentzian indefinite fold mapping nor 
%the mapping 
%$(x_0, \ldots, x_n)\mapsto (x_0, x_1\ldots, x_{k-1}, \sum_{i=k}^{n}x_i^2)$ is 
%$\mathcal{A}$-equivalent to the normal form of definite fold mapping.    
%Hence, we have the following corollary.   
\noindent 
%Note that the mapping $(x_0, \ldots, x_n)\mapsto (x_0, x_1\ldots, x_{k-1}, \sum_{i=k}^{n}x_i^2)$, 
%appeared in Theorem \ref{theorem 1} (1) (b), 
%is $\mathcal{A}$-equivalent to the normal form of 
%Lorentzian definite fold mapping $\Phi_k$.    
Combining Theorem \ref{theorem 1} and Proposition \ref{proposition 1} yields 
the following:  
\begin{corollary}\label{corollary 2}
\begin{enumerate}
\item 
Let $k, n$ be two positive integers satisfying $ k< n$ and 
let $p_0, \ldots, p_k$ belonging to $\mathbb{R}^{1,n}$ be 
$(k+1)$ points in general position. 
%Given $2$ integers $k, n$ satisfying $1\le k <n$ and 
%$(k+1)$ points $p_0, \ldots, p_k\in \mathbb{R}^{1,n}$ in general position, 
Then, the Lorentzian distance-squared mapping $L_{(p_0, \ldots, p_k)}$ 
is $\mathcal{A}$-equivalent to the distance-squared mapping 
$D_{(p_0, \ldots, p_k)}$ 
if and only if 
the recognition subspace 
$V(p_0, \ldots, p_k)$ is time-like. 
\item 
Let $n$ be a positive integer and 
let $p_0, \ldots, p_n\in \mathbb{R}^{1,n}$ be $(n+1)$ points in general position. 
%Given $2$ integers $k, n$ satisfying $1\le k <n$ and 
%$(k+1)$ points $p_0, \ldots, p_k\in \mathbb{R}^{1,n}$ in general position, 
Then, the Lorentzian distance-squared mapping $L_{(p_0, \ldots, p_n)}$ 
is $\mathcal{A}$-equivalent to the distance-squared mapping 
$D_{(p_0, \ldots, p_n)}$ 
if and only if 
the recognition subspace 
$V(p_0, \ldots, p_n)$ is time-like or space-like. 
\item 
Let $k, n$ be two positive integers satisfying $n<k$ and 
let $p_0, \ldots, p_{k}\in \mathbb{R}^{1,n}$ be $(k+1)$ points such that 
the $(n+2)$ points $p_0, \ldots, p_{n+1}$ are in general position.    
%Let $k, n$ be two positive integers satisfying $n<k$ and 
%let $p_0, \ldots, p_{k}\in \mathbb{R}^{1,n}$ be $(n+2)$ points such that 
%$p_0, \ldots, p_{n+1}$ are in general position.    
Then, 
the Lorentzian distance-squared mapping $L_{(p_0, \ldots, p_k)}$ 
is always $\mathcal{A}$-equivalent to the distance-squared mapping 
$D_{(p_0, \ldots, p_k)}$.   
\end{enumerate}
\end{corollary}
\begin{example}\label{example 1.3}
{\rm Let $\phi, \widetilde{\phi}, \psi : 
\mathbb{R}^{3}\to \mathbb{R}^2$ be  
polynomial mappings defined as follows:    
\begin{eqnarray*}
\phi(x, y, z) & = & 
\left(-x^2+y^2+z^2, -(x+1)^2+(y+2)^2+(z+1)^2\right),  \\  
\widetilde{\phi}(x, y, z) & = & 
\left(x^2+y^2+z^2, (x+1)^2+(y+2)^2+(z+1)^2\right), \\ 
\psi(x, y, z) & = & 
\left(-x^2+y^2+z^2, -(x+2)^2+(y+1)^2+(z+1)^2\right). 
\end{eqnarray*} 
By Corollary \ref{corollary 2}, 
$\phi$ is not $\mathcal{A}$-equivalent to $\widetilde{\phi}$, while $\psi$ is $\mathcal{A}$-equivalent to 
$\widetilde{\phi}$.  
%Thus, it can be seen that all fibers of $g$ are compact, while all 
%fibers of  $f$ and $h$ are non-compact.    Also, it can be seen that 
%the mappings $f$ and $h$ are not surjective, though the topological closure of $f(\mathbb{R}^{1,2})$ is 
%$\mathbb{R}^2$; on the other hand, $g$ is surjective.  
}
\end{example}
%It is clear that the Lorentzian distance-squared function $l_p$ is never $\mathcal{A}$-equivalent to 
%the distance-squared function $d_p$.    Nevertheless, by taking at least two points $p_0, p_1$ so that the recognition %subspace 
%$V(p_0, p_1)$ is time-like, $L_{(p_0, p_1)}$ is always $\mathcal{A}$-equivalent to $D_{(p_0, p_1)}$ by Corollary %\ref{corollary 3}.   
%%%%%%%%%%%%%%%%%%%%%%%%%%%%%%%%%%%%%%%%%%%%%%%%%% 
\par 
\medskip 
In Section 2, preliminaries for the proof of Theorem \ref{theorem 1} are given. 
Theorem \ref{theorem 1} is proved in Section 3. 
Finally, as an appendix, a recognizable classification of degenerate 
Lorentzian distance-squared mappings is given in Section 4.    
%%%%%%%%%%%%%%%%%%%%%%%%%%%%%%%%%%%%%%%%%%%%%%%%%% 
%%%%%%%%%%%%%%%%%%%%%%%%%%%%%%%%%%%%%%%%%%%%%%%%%% 
%%% %%%%%%%%%%%%%%%%%%%%%%%%%%%%%%%%%%%%%%%%%%%%%%% 
%%%%%%%%%%%%%%%%%%%%%%%%%%%%%%%%%%%%%%%%%%%%%%%%%% 
\section{Preliminaries}\label{preliminaries}
%%% %%%%%%%%%%%%%%%%%%%%%%%%%%%%%%%%%%%%%%%%%%%%%%% 
%%%%%%%%%%%%%%%%%%%%%%%%%%%%%%%%%%%%%%%%%%%%%%%%%%%%
\begin{lemma}\label{lemma 1}
%%%%%%%%%%%%%%%%%%%%%%%%%%%%%%%%%%%%%%%%%%%%%%%%%%%%%%%%%
The likeness of a vector subspace of 
$\mathbb{R}^{1,n}$ is invariant under Lorentz transformations.
%Let $V$ be an $n$-dimensional vector subspace of $\mathbb{R}^{1,n}$ defined by 
%$-x_0+\alpha_1 x_1 +\cdots +\alpha_n x_n=0$. 
%Then, the following hold:
\end{lemma}
Lemma \ref{lemma 1} clearly holds.
%%%%%%%%%%%%%%%%%%%%%%%%%%%%%%%%%%%%%%%%%%%%%%%%%%%%%%%%%%%%%
%%%%%%%%%%%%%%%%%%%%%%%%%%%%%%%%%%%%%%%%%%%%%%%%%%%%%%%%%%%%%
\begin{lemma}\label{lemma 2}
%%%%%%%%%%%%%%%%%%%%%%%%%%%%%%%%%%%%%%%%%%%%%%%%%%%%%%%%%%%%%
Let ${\bf v}_1, \ldots, {\bf v}_n$ be the vectors of $\mathbb{R}^{1,n}$ 
defined as follows:
\begin{eqnarray*}
{\bf v}_1&=&(\alpha_1,1,0,\ldots,0),\\
{\bf v}_2&=&(\alpha_2,0,1,0,\ldots,0),\\
\vdots \\
{\bf v}_k&=&(\alpha_k,\underbrace{0,\ldots,0}_{(k-1)-tuples},1,0,\ldots,0),\\
{\bf v}_{k+1}&=&(\underbrace{0,\ldots,0}_{(k+1)-tuples},1,0,\ldots,0), \\
\vdots \\
{\bf v}_n&=&(0,\ldots,0,1),\\
\end{eqnarray*}
where $1\leq k<n$. 
Let $V$ be the $k$-dimensional vector subspace of $\mathbb{R}^{1,n}$ defined by 
$V=\sum_{i=1}^k\mathbb{R}{\bf v}_i$, and 
let $\widetilde{V}$ be the $n$-dimensional vector subspace of 
$\mathbb{R}^{1,n}$ defined by 
$\widetilde{V}=\sum_{i=1}^n\mathbb{R}{\bf v}_i$. 
Then, the following hold:
\begin{enumerate}
\item $\widetilde{V}$ is time-like if and only if  
$V$ is time-like.  
\item $\widetilde{V}$ is space-like if and only if  
$V$ is space-like.   
\item $\widetilde{V}$ is light-like if and only if 
$V$ is light-like.   
\end{enumerate}
%Then, the likeness of the recognition subspace $V$ is 
%the same as the likeness of the recognition subspace $\widetilde{V}$.
\end{lemma}
\begin{proof}
By definition, $V$ is either time-like or space-like or light-like. 
Thus, in order to prove Lemma \ref{lemma 2}, 
it is sufficient to show only the \lq\lq if parts\rq\rq of $(1)$, $(2)$, and $(3)$. 

Suppose that $V$ is time-like. Then, since $V\subset \widetilde{V}$, 
$\widetilde{V}$ is also time-like by Definition 1.

For any vector $\sum_{i=1}^kr_i{\bf v}_i\in V$ and 
$\sum_{i=1}^nr_i{\bf v}_i\in \widetilde{V}$, 
we have the following:
\begin{eqnarray}
\left\langle\sum_{i=1}^kr_i{\bf v}_i, \sum_{i=1}^kr_i{\bf v}_i\right\rangle &=&
-\left(\sum_{i=1}^kr_i\alpha_i\right)^2+\sum_{i=1}^kr_i^2\\
\left\langle\sum_{i=1}^nr_i{\bf v}_i, \sum_{i=1}^nr_i{\bf v}_i\right\rangle &=&
-\left(\sum_{i=1}^kr_i\alpha_i\right)^2+\sum_{i=1}^nr_i^2\nonumber 
\\&=&
-\left(\sum_{i=1}^kr_i\alpha_i\right)^2+\sum_{i=1}^kr_i^2+\sum_{i=k+1}^nr_i^2.
\end{eqnarray}

Suppose that $V$ is space-like. Then, by Definition 1, 
any nonzero vector in $V$ is space-like. 
Thus, by (1) and (2), every nonzero vector 
in $\widetilde{V}$ is also space-like.
 
Suppose that $V$ is light-like. 
Then, by Definition 1, $V$ has a nonzero light-like vector ${\bf v}$. 
The vector ${\bf v}$ is also in 
$\widetilde{V}$. Since $V$ has no time-like vectors, 
by (1) and (2), 
it follows that $\widetilde{V}$ has no time-like vectors.
\end{proof}
%%%%%%%%%%%%%%%%%%%%%%%%%%%%%%%%%%%%%%%%%%%%%%%%%%%%%%%%%%%%%
%%%%%%%%%%%%%%%%%%%%%%%%%%%%%%%%%%%%%%%%%%%%%%%%%%%%%%%%%%%%%
\begin{lemma}\label{lemma 3}
%%%%%%%%%%%%%%%%%%%%%%%%%%%%%%%%%%%%%%%%%%%%%%%%%%%%%%%%%%%%%
Given $(\alpha_1,\ldots,\alpha_n)\in \mathbb{R}^n$, 
let $V$ be the $n$-dimensional vector subspace of 
$\mathbb{R}^{1,n}$ defined by 
$-x_0+\alpha_1 x_1 +\cdots +\alpha_n x_n=0$. 
Then, the following hold:
\begin{enumerate}
\item $\sum_{i=1}^n\alpha_i^2 > 1$ if and only if  
$V$ is time-like.  
\item $\sum_{i=1}^n\alpha_i^2<1$ if and only if  
$V$ is space-like.   
\item $\sum_{i=1}^n\alpha_i^2=1$ if and only if  
$V$ is light-like.   
\end{enumerate}
\begin{proof}
Let $H$ be the horizontal hyperplane $\{(1, x_1, \ldots, x_n)\; |\; 
x_i\in \mathbb{R} \}$. 
Set 
$V_1=H\cap V$.    
%\par 
Suppose that $V_1=\emptyset$.      Then, the defining equation of $V$ is 
$-x_0=0$.  Thus, $\sum_{i=1}^n\alpha_i^2=0$ and $V$ is space-like.   
\par 
Next, suppose that $V_1\ne \emptyset$.      
Then, $\sum_{i=1}^n\alpha_i^2\ne 0$.    
Let $q$ be the point $(1, 0, \ldots, 0)$ and 
let $S_+^{n-1}$ be the light cone hypersurface $H\cap LC$.  
Then, it is clear that the Euclidean distance between $q$ and any point 
$x\in S_+^{n-1}$ is $1$.   
In order to complete the proof, 
it is sufficient to show that 
$1/\sqrt{\sum_{i=1}^n\alpha_i^2}$ is the Euclidean distance between $q$ and 
$V_1$.    
%\par 
Since $-x_0+\alpha_1 x_1 +\cdots +\alpha_n x_n=0$ is a defining equation of 
$V$, $V_1$ is defined by 
$-1+\alpha_1 x_1 +\cdots +\alpha_n x_n=0$ in $H$.   Hence, the Euclidean distance 
between  $q$ and $V_1$ is 
$1/\sqrt{\sum_{i=1}^n\alpha_i^2}$.  
\end{proof} 
\end{lemma}
%%%%%%%%%%%%%%%%%%%%%%%%%%%%%%%%%%%%%%%%%%%%%%%%%%%%
%%%%%%%%%%%%%%%%%%%%%%%%%%%%%%%%%%%%%%%%%%%%%%%%
\section{Proof of Theorem \ref{theorem 1}}\label{section 3}
%%%%%%%%%%%%%%%%%%%%%%%%%%%%%%%%%%%%%%%%%%%%%%%%%%
The proof needs more elaborated and more careful constructions 
of affine transformations of the source space and 
quadratic transformations of the target space 
than ones for the proof of Proposition \ref{proposition 1} given in \cite{ichikinishimura}.   
Let $(X_0, X_1, \ldots, X_k)$ be the standard coordinate of $\mathbb{R}^{k+1}$.    
%%%%%%%%%%%%%%%%%%%%%%%%%%%%%%%%%%%%%%%%%%%%%%%%%%
\subsection{Proof of (1) of Theorem \ref{theorem 1}}
%%%%%%%%%%%%%%%%%%%%%%%%%%%%%%%%%%%%%%%%%%%%%%%%%% 
%There are the three four cases on 
%$V(p_0, \ldots, p_k)$.   
%These $3$ cases correspond to 
%the cases stated in (a), (b) and (c) of (1) of Theorem \ref{theorem 1} respectively.   
%For (2) of Theorem \ref{theorem 1},  
%the cases (1) and (2) correspond to (a) and (3) corresponds to (b).      
It is easily seen that any two among $\Phi_k, \Psi_k$   
%$(x_0, \ldots, x_n)\mapsto (x_0, \ldots, x_{k-1}, \sum_{i=k}^{n}x_i^2)$  
and  $(x_0, \ldots, x_n)\mapsto 
(x_1, \ldots, x_k, x_0x_1+\sum_{i=k+1}^n x_i^2)$      
are not $\mathcal{A}$-equivalent.       
Moreover, by definition, $V(p_{0}, \ldots, p_k)$ is either time-like or space-like or light-like.    
%And also, it is clear that any two among $\Phi_n, \Psi_n$  
%$(x_0, \ldots, x_n)\mapsto (x_0, x_1\ldots, x_{n-1}, x_n^2)$ 
%and 
%$(x_0, \ldots, x_n)\mapsto (x_1, \ldots, x_n, x_0x_1)$ 
%are not 
%$\mathcal{A}$-equivalent.          
Thus, in order to prove $(1)$ of Theorem \ref{theorem 1}, 
it is sufficient to show only the \lq\lq if parts \rq\rq of $(1)$ of 
Theorem \ref{theorem 1}.         
Set $p_i=(p_{i0}, p_{i1}, \ldots, p_{in})$ $(0\le i\le k)$.    
 \par 
\medskip 
%%%%%%%%%%%%%%%%%%%%%%%%%%%%%%%%%%%%%%%%%%%%%%%%%% 
%%%%%%%%%%%%%%%%%%%%%%%%%%%%%%%%%%%%%%%% 
\subsubsection{The generic case}\label{The generic case}
\qquad 
%%%%%%%%%%%%%%%%%%%%%%%%%%%%%%%%%%%%%%%% 
We first show the \lq\lq if parts\rq\rq of $(1)$ of Theorem \ref{theorem 1} in the case 
that $V(p_0, \ldots, p_k)\cap T= \{0\}$, where $T$ is the time axis 
$\{(x_0, 0, \ldots, 0)\; |\; x_0\in \mathbb{R}\}$.        
There are four steps.    
\par  
\smallskip 
\underline{STEP 1.}\qquad The purpose of Step 1 is to 
remove the redundant quadratic terms in $k$ components.   
In order to do so, we require the affine transformation 
of the target space $H_1 : \mathbb{R}^{k+1}\rightarrow \mathbb{R}^{k+1}$ defined by 
\begin{eqnarray*}
{}&{}&H_1(X_0,X_1,\ldots, X_k)\\
{}&=&\biggl(\frac{1}{2}\Bigl(X_0-X_1-(p_{00}-p_{10})^2+\sum_{i=1}^n (p_{0i}-p_{1i})^2\Bigl),\ldots, \\ 
{ } & { } & \qquad 
\frac{1}{2}\Bigl(X_0-X_k-(p_{00}-p_{k0})^2+\sum_{i=1}^n (p_{0i}-p_{ki})^2\Bigl),X_0\biggl).   
\end{eqnarray*}
The composition of $L_{(p_0,\ldots,p_k)}$ and $H_1$ has the following form: 
\begin{eqnarray*}
&{}&(H_1\circ L_{(p_0,\ldots,p_k)})(x_0,x_1,\ldots,x_n)\\
&=&
\left(
-\left(p_{10}-p_{00}\right)\left(x_0-p_{00}\right)
+\sum_{i=1}^n (p_{1i}-p_{0i})(x_i-p_{0i}),\ldots, 
%\sum_{j=0}^n (p_{k j}-p_{0j})(x_j-p_{0j}),
\right. \\ 
{ } & { } & 
\qquad  
-\left.\left(x_0-p_{00}\right)^2+\sum_{i=1}^n(x_i-p_{0i})^2\right).   
\end{eqnarray*}
\par 
\smallskip 
\underline{STEP 2.}\qquad 
The purpose of Step 2 is to reduce the first $k$ components to linear functions.   
In order to do so, we require the affine transformation of the source space 
$H_2 : \mathbb{R}^{1,n}\rightarrow \mathbb{R}^{1,n}$ defined by
\[
H_2(x_0, x_1,\ldots,x_n)=(-x_0+p_{00},x_1+p_{01},\ldots,x_n+p_{0n}).
\]
The composition of $H_1\circ L_{(p_0,\ldots,p_k)}$ and $H_2$ has the 
following form:   
\begin{eqnarray*}
{}&{}&(H_1\circ L_{(p_0,\ldots,p_k)}\circ H_2)(x_0,x_1,\ldots,x_n)\\
{}&=&\left(\sum_{i=0}^n (p_{1i}-p_{0i})x_i,\ldots,\sum_{i=0}^n (p_{ki}-p_{0i})x_i,\, 
-x_0^2+\sum_{i=1}^nx_i^2\right)\\
{}&=&
\begin{pmatrix}
x_0,&x_1,&\cdots &x_n
\end{pmatrix}
\left(
\begin{array}{cccr}
p_{10}-p_{00}&\cdots &p_{k0}-p_{00}&-x_0\\
p_{11}-p_{01}&\cdots &p_{k1}-p_{01}&x_1\\
\vdots & & \vdots &\vdots \\
\vdots &   & \vdots &\vdots \\
p_{1n}-p_{0n}&\cdots &p_{k n}-p_{0n}&x_n
\end{array}
\right).
\end{eqnarray*}
Set 
\begin{eqnarray*}
A=\begin{pmatrix}
p_{10}-p_{00}&\cdots &p_{k0}-p_{00}\\
p_{11}-p_{01}&\cdots &p_{k1}-p_{01}\\
\vdots & & \vdots \\
\vdots &   & \vdots \\
p_{1n}-p_{0n}&\cdots &p_{k n}-p_{0n}
\end{pmatrix}
= 
\left(
\begin{array}{ccc}
p_{10}-p_{00}&\cdots &p_{k0}-p_{00}  \\
 \hline 
{ } & { } & { }  \\
{ } & { } & { } \\
{ } & \mbox{\smash{\LARGE $\widetilde{A}$}}& { } \\
{ } & { } & { } \\ 
{ } & { } & { } \\ 
\end{array}
\right).
\end{eqnarray*}
Since $(k+1)$-points $p_0,\ldots,p_k$ are in general position, 
it is clear that the rank of $(n+1)\times k$ matrix $A$ is $k$.    
Moreover, since $V(p_0, \ldots, p_k)\cap T=\{0\}$,  
the $n\times k$ matrix $\widetilde{A}$ has the same rank $k$.
There exists a $k\times k$ regular matrix $B$ such that the set of 
column vectors of $\widetilde {A}B$ is the subset of an orthonormal 
basis of $\mathbb{R}^{n}$. Set 
\begin{eqnarray*}
\widetilde {A}B=\begin{pmatrix}
    \alpha _{1,1}  & \cdots &\alpha _{1,k}\\
 \vdots           & \ddots    &  \vdots  \\
   \alpha _{n,1}  & \cdots &\alpha _{n,k}\\
\end{pmatrix}.
\end{eqnarray*}
\smallskip 
%%%%%%%%%%%%%%%%%%%%%%%%%%%%%%%%%%%%%%%%%%%%%%%%%%%  
\underline{STEP 3.}\qquad    
The purpose of this step is to 
reduce the first $k$ components, which are linear functions, 
to coordinate functions $x_1, \ldots, x_k$, preserving 
the Lorentzian distance-squared function 
$-x_0^2+\sum_{i=1}^nx_i^2$ having the form 
$\xi (x_0,x_1,\ldots,x_k)+\sum_{i=k+1}^nx_i^2$. 
In order to do so, we construct the linear transformation of the target space 
$H_3 : \mathbb{R}^{k+1}\rightarrow \mathbb{R}^{k+1}$ below 
and the linear transformations of the source space 
$H_4, H_5 : \mathbb{R}^{1, n}\rightarrow \mathbb{R}^{1, n}$ below.    
\par 
Firstly, the construction of $H_3$ is given.   
Let $H_3 : \mathbb{R}^{k+1}\rightarrow \mathbb{R}^{k+1}$ be 
the linear transformation 
of $\mathbb{R}^{k+1}$ 
defined by
\begin{eqnarray*}
{}&{}&H_3(X_0,X_1,\ldots,X_k )=
%\begin{pmatrix}
\left( X_0, X_1, \ldots , X_k\right)
%\end{pmatrix}
\left(
\begin{array}{ccc|c}
&&&  0\\ 
&\mbox{\smash{\Large $B$}}&&\vdots\\
&&&0\\ \hline
0&\cdots &0&1
\end{array}
\right).
\end{eqnarray*}
The composition of $H_1\circ L_{(p_0,\ldots,p_k)}\circ H_2$ and $H_3$ is given by 
\begin{eqnarray*}
{}&{}&\left(H_3\circ H_1\circ L_{(p_0,\ldots,p_k)}\circ H_2\right)\left(x_0,x_1,\ldots,x_n\right)\\
{}&=&
\left( 
x_0, x_1,\ldots , x_n 
\right)
\left(
\begin{array}{ccc|r}
&&&  -x_0\\ 
&&& x_1\\
&\mbox{\smash{\Large $A$}}&&\vdots \\
&&&\vdots \\
&&&x_n
\end{array}
\right)
\left(
\begin{array}{ccc|c}
&&&  0\\ 
&\mbox{\smash{\Large $B$}}&&\vdots\\
&&&0\\ \hline
0&\cdots &0&1
\end{array}
\right)\\
{}&=&
%\begin{pmatrix}
\left(x_0, x_1, \ldots , x_n\right)
%\end{pmatrix}
\left(
\begin{array}{cccr}
     \alpha _{0,1}  & \cdots &\alpha _{0,k} & -x_0\\
       \alpha _{1,1}    &  \cdots   &  \alpha _{1,k}    & x_1\\
        \vdots             & \ddots &  \vdots   & \vdots  \\
  \alpha _{n,1}  & \cdots &\alpha _{n.k} & x_{n} \\
\end{array}
\right) .\\ 
\end{eqnarray*}

Set ${\bf a}_i=(\alpha_{0,i},\alpha_{1,i},\ldots,\alpha_{n,i})$, 
${\bf \widetilde{a}}_i=(\alpha_{1,i},\ldots,\alpha_{n,i})$ $(1\leq i\leq k)$. 
Note that $V(p_0,\ldots ,p_k)=\sum_{i=1}^k\mathbb{R}{\bf a}_i$. 
We can choose column vectors ${\bf \widetilde{a}}_{k+1},\ldots ,{\bf \widetilde{a}}_n$ such that 
the set $\{{\bf \widetilde{a}}_1,\ldots ,{\bf \widetilde{a}}_n\}$ is an 
orthonormal basis of $\mathbb{R}^n$. 
Put $C=({}^t{\bf \widetilde{a}_1},\ldots ,{}^t{\bf \widetilde{a}}_n)$. 
Note that the matrix $C$ is an $n\times n$ orthogonal matrix. 
%%%%%%%%%%%%%%%%%%%%%%%%%%%%%%%%%%%%%%%%%%%%%%
Let $H_4 : \mathbb{R}^{1,n}\rightarrow \mathbb{R}^{1,n}$ be 
the linear isomorphism defined by 
\begin{eqnarray*}
{}&{}&H_4(x_0,x_1,\ldots,x_n)=
\left( 
x_0, x_1,\ldots , x_n 
\right)
\left(
\begin{array}{c|ccc}
1&0&\cdots &0\\
\hline 
0&&&\\
\vdots &&\mbox{\smash{\Large $C$}}&\\
0&&&  
\end{array}
\right).
\end{eqnarray*}
Note that $H_4$ and $H_4^{-1}$ are Lorentz transformations. 
The composition of $H_3\circ H_1\circ L_{(p_0,\ldots,p_k)}\circ H_2$ and $H_4$
is as follows:
\begin{eqnarray*}
{}&{}&(H_3\circ H_1\circ L_{(p_0,\ldots,p_k)}\circ H_2\circ H_4)(x_0,x_1,\ldots,x_n)\\
%&=&(x_0,x_1,\ldots,x_n)
%\left(
%\begin{array}{cccr}
%\alpha _{0,1}  & \cdots &\alpha _{0,k}   & -x_0\\
%1  &      & \bigzerol & x_1\\
% &       \ddots        &  &   \vdots \\
% &                 &  1&   \vdots \\
% &     \bigzerol            &  &   \vdots \\
%  &             &   &  x_n \\
%\end{array}
%\right) \\ 
{}&=&\left(\alpha_{0,1}x_0+x_1,\ldots,\alpha _{0,k}x_0+x_k, -x_0^2
+\sum_{i=1}^nx_i^2\right).
\end{eqnarray*}

Then, we have 
\begin{eqnarray*}
{}&{}&(x_0,\alpha _{0,1}x_0+x_1,\ldots,\alpha _{0,k}x_0+x_k, x_{k+1},\ldots,x_n)\\
&=&(x_0,x_1,\ldots,x_n)
\begin{pmatrix}
1 &   \alpha _{0,1}  & \cdots &\cdots &\alpha _{0,k}   & &\\
& 1  &0   &   \cdots   & 0 && &\\
 &               & \ddots &   \ddots  & \vdots &\bigzerol &\\
 &               &  &   \ddots  & 0 &&\\
  &               &\bigzerol   &     &1& &\\
  &&&& &\ddots &\\
&&&&&&1
\end{pmatrix}.
\end{eqnarray*}
Let $D$ denote the last $(n+1)\times (n+1)$ matrix. 
For any $i$ $(2\leq i\leq n+1)$, let ${\bf b_i}$ 
denote the $i$-th column vector of $D$. 
Since $H_4^{-1}$ is a Lorentz transformation, by Lemma \ref{lemma 1}, 
the likeness of $\sum_{i=1}^k\mathbb{R}{}^t{\bf b_i}$ is 
the same as the likeness of $V(p_0,\ldots,p_k)$. 
Moreover, by Lemma \ref{lemma 2}, the likeness of 
$\sum_{i=1}^k\mathbb{R}{}^t{\bf b_i}$ is the same as 
the likeness of $\sum_{i=1}^n\mathbb{R}{}^t{\bf b_i}$. 
Thus, the likeness of $V(p_0,\ldots,p_k)$ is the 
same as the likeness of $\sum_{i=1}^n\mathbb{R}{}^t{\bf b_i}$. 
Note that the $n$-dimensional vector subspace 
$\sum_{i=1}^n\mathbb{R}{}^t{\bf b_i}\subset \mathbb{R}^{1,n}$ 
is defined by 
$-x_0+\alpha_{0,1}x_1+\cdots+\alpha_{0,k}x_k=0$. 
Therefore, by Lemma \ref{lemma 3}, we have the following: 
\begin{lemma}\label{lemma 4}
\begin{enumerate}
\item $\sum_{i=1}^k\alpha_{0,i}^2 > 1$ if and only if  
$V(p_0, \ldots, p_k)$ is time-like.  
\item $\sum_{i=1}^k\alpha_{0,i}^2<1$ if and only if  
$V(p_0, \ldots, p_k)$ is space-like.   
\item $\sum_{i=1}^k\alpha_{0,i}^2=1$ if and only if  
$V(p_0, \ldots, p_k)$ is light-like.   
\end{enumerate}
\end{lemma}
Since $D$ is a regular matrix, we get 
\begin{eqnarray*}
D^{-1}=
\begin{pmatrix}
1 &   -\alpha _{0,1}  & \cdots &\cdots &-\alpha _{0,k}   & &\\
& 1  &0   &   \cdots   & 0 && &\\
 &               & \ddots &   \ddots  & \vdots &\bigzerol &\\
 &               &  &   \ddots  & 0 &&\\
  &               &\bigzerol   &     &1& &\\
  &&&& &\ddots &\\
&&&&&&1
\end{pmatrix}.
\end{eqnarray*}
Let $H_5 : \mathbb{R}^{1,n}\rightarrow \mathbb{R}^{1,n}$ 
be the linear isomorphism defined by 
\[
H_5(x)=xD^{-1}.
\]
The composition of $H_3\circ H_1\circ L_{(p_0,\ldots,p_k)}\circ H_2\circ H_4$ and $H_5$ is as follows:
\begin{eqnarray*}
{}&{}&(H_3\circ H_1\circ L_{(p_0,\ldots,p_k)}\circ H_2\circ H_4\circ H_5)(x_0,x_1,\ldots,x_n)\\
{}&=&\biggl(x_1,x_2,\ldots,x_{k},
-x_0^2
+\sum_{i=1}^{k}(-\alpha_{0,i}x_0+x_i)^2+
\sum_{i=k+1}^n x_i^2\biggl).
\end{eqnarray*}  
\par 
\smallskip 
\underline{STEP 4.}\qquad 
This is the last step. By transforming the last component of the mapping 
$H_3\circ H_1\circ L_{(p_0,\ldots,p_k)}\circ H_2\circ H_4\circ H_5$ with 
respect to the variable $x_0$, we have the following:    
\begin{eqnarray*}
{}&{}&-x_0^2
+\sum_{i=1}^{k}(-\alpha_{0,i}x_0+x_i)^2+
\sum_{i=k+1}^n x_i^2\\
{}&=&\left(-1+\sum_{i=1}^k\alpha_{0,i}^2\right)x_0^2
-2 x_0\sum_{i=1}^k\alpha_{0,i}x_i
+\sum_{i=1}^nx_i^2.    
\end{eqnarray*}
\par 
Firstly, the cases (a) and (b) of $(1)$ of Theorem \ref{theorem 1} are proved. 
By Lemma \ref{lemma 4}, 
it follows that $-1+\sum_{i=1}^k\alpha_{0,i}^2 \ne 0$.          
Thus, by completing the square with respect to the variable 
$x_0$, we have the following: 
\begin{eqnarray*}
{}&{}&\left(-1+\sum_{i=1}^k\alpha_{0,i}^2\right)x_0^2
-2 x_0\sum_{i=1}^k\alpha_{0,i}x_i
+\sum_{i=1}^nx_i^2\\
{}&=&\left(-1+\sum_{i=1}^k\alpha_{0,i}^2\right)
\left(x_0-\frac{\sum_{i=1}^k\alpha_{0,i}x_i}
{-1+\sum_{i=1}^k\alpha_{0,i}^2}\right)^2 %\\
%{ } & { } & 
%\qquad - 
-\frac{\left(\sum_{i=1}^k\alpha_{0,i}x_i\right)^2}
{-1+\sum_{i=1}^k\alpha_{0,i}^2} 
+
\sum_{i=1}^nx_i^2.\\  
\end{eqnarray*}
Let $H_6:\mathbb{R}^{k+1}\rightarrow \mathbb{R}^{k+1}$ 
be the diffeomorphism defined by 
\begin{eqnarray*}
{}&{}&H_6 (X_0,X_1,\ldots,X_k)\\
{}&=&\left(X_0,X_1,\ldots,X_{k-1},X_k+
\frac{\left(\sum_{i=1}^k\alpha_{0,i}X_{i-1}\right)^2}
{-1+\sum_{i=1}^k\alpha_{0,i}^2} 
-
\sum_{i=1}^kX_{i-1}^2\right).\\ 
\end{eqnarray*}
The composition of $H_6$ and $H_3\circ H_1\circ L_{(p_0,\ldots,p_k)}\circ H_2\circ H_4\circ H_5$ 
is as follows: 
\begin{eqnarray*}
{}&{}&(H_6\circ H_3\circ H_1\circ L_{(p_0,\ldots,p_k)}\circ H_2\circ H_4\circ H_5)(x_0,x_1,\ldots, x_n)\\
{}&=&\left(x_1,\ldots,x_k,\left(-1+\sum_{i=1}^k\alpha_{0,i}^2\right)
\left(x_0-\frac{\sum_{i=1}^k\alpha_{0,i}x_i}
{-1+\sum_{i=1}^k\alpha_{0,i}^2}\right)^2
+
\sum_{i=k+1}^nx_i^2\right).
\end{eqnarray*}
Let $H_7:\mathbb{R}^{1,n}\rightarrow \mathbb{R}^{1,n}$ 
be the linear isomorphism defined by 
\begin{eqnarray*}
{}&{}&H_7(x_0,x_1,\ldots,x_n)\\
{}&=&\left(
\frac{x_0}{\sqrt {\mid -1+\sum_{i=1}^k\alpha_{0,i}^2 \mid }}
+
\frac{\sum_{i=1}^k\alpha_{0,i}x_i}{-1+\sum_{i=1}^k\alpha_{0,i}^2},x_1,\ldots,x_n\right).
\end{eqnarray*}
By Lemma \ref{lemma 4}, if $V(p_0,\ldots,p_k)$ is time-like, the composition of  
$H_6\circ H_3\circ H_1\circ L_{(p_0,\ldots,p_k)}\circ H_2\circ H_4\circ H_5$ and 
$H_7$ is as follows:
\begin{eqnarray*}
{}&{}&(H_6\circ H_3\circ H_1\circ L_{(p_0,\ldots,p_k)}\circ H_2\circ H_4\circ H_5\circ H_7)(x_0,x_1,\ldots,x_n)\\
{}&=&\left(x_1,\ldots,x_k,x_0^2+\sum_{i=k+1}^nx_i^2\right),
\end{eqnarray*}
and if $V(p_0,\ldots,p_k)$ is space-like, the composition of them is as follows:
\begin{eqnarray*}
{}&{}&(H_6\circ H_3\circ H_1\circ L_{(p_0,\ldots,p_k)}\circ H_2\circ H_4\circ H_5\circ H_7)(x_0,x_1,\ldots,x_n)\\
{}&=&\left(x_1,\ldots,x_k,-x_0^2+\sum_{i=k+1}^nx_i^2\right).
\end{eqnarray*}
%%%%%%%%%%%%%%%%%%%%%%%%%%%%%%%%%%%%%%%%%%%%%%%%%%%%%%%%%%%%

Next, the case (c) of $(1)$ of Theorem \ref{theorem 1} is proved.
By Lemma \ref{lemma 4}, we have the following:
\begin{eqnarray*}
{}&{}&(H_3\circ H_1\circ L_{(p_0,\ldots,p_k)}\circ H_2\circ H_4\circ H_5)(x_0,x_1,\ldots,x_n)\\
{}&=&\left(x_1,\ldots,x_k,-2x_0\sum_{i=1}^k\alpha_{0,i}x_i+\sum_{i=1}^nx_i^2\right).
\end{eqnarray*}
Let $H_6':\mathbb{R}^{k+1}\rightarrow \mathbb{R}^{k+1}$ be 
the diffeomorphism defined by 
\begin{eqnarray*}
{}&{}&H_6' (X_0,X_1,\ldots,X_k)\\
{}&=&\left(X_0,X_1,\ldots,X_{k-1},X_k-\sum_{i=1}^kX_{i-1}^2\right).\\ 
\end{eqnarray*}
The composition of 
$H_3\circ H_1\circ L_{(p_0,\ldots,p_k)}\circ H_2\circ H_4\circ H_5$ and $H_6'$ 
is as follows:
\begin{eqnarray*}
{}&{}&(H_6'\circ H_3\circ H_1\circ L_{(p_0,\ldots,p_k)}\circ H_2\circ H_4\circ H_5)
(x_0,x_1,\ldots,x_n)\\
{}&=&\left(x_1,\ldots,x_k,-2x_0\sum_{i=1}^k\alpha_{0,i}x_i+\sum_{i=k+1}^nx_i^2\right).
\end{eqnarray*}
Since $-1+\sum_{i=1}^k\alpha_{0,i}^2=0$, 
there must exist a $j$ $(1\leq j\leq k)$ 
such that $\alpha_{0,j}\not=0$. 
By taking a linear transformation of the source space if necessary, 
without loss of generality, 
we may assume that $\alpha_{0,1}\not=0$. 
Then, we have 
\begin{eqnarray*}
{}&{}&\left(x_0,-2\sum_{i=1}^k\alpha_{0,i}x_i, x_2, \ldots,x_n\right)\\
{}&=&(x_0,x_1,\ldots,x_n)
\begin{pmatrix}
1&0&\cdots & \cdots  &\cdots &\cdots &\cdots &0\\
0&-2\alpha _{0,1}&&&&&&\\
\vdots &\vdots &1&&&\bigzerol &&\\
\vdots &\vdots &&\ddots&&&&\\
\vdots &-2\alpha _{0,k}&&&\ddots & &&\\
\vdots &0&&\bigzerol &&\ddots& &\\
\vdots &\vdots && &&&\ddots &\\
0&0&&&&&&1
\end{pmatrix}.
\end{eqnarray*}
Let $M$ denote the last $(n+1)\times (n+1)$ matrix.
Then, we have the following: 
\begin{eqnarray*}
M^{-1}=
\begin{pmatrix}
1&0&\cdots & \cdots  &\cdots &\cdots &\cdots &0\\
0&-\frac{1}{2\alpha _{0,1}}&&&&&&\\
\vdots &-\frac{\alpha_{0,2}}{\alpha_{0,1}}&1&&&\bigzerol &&\\
\vdots &\vdots &&\ddots&&&&\\
\vdots &-\frac{\alpha_{0,k}}{\alpha_{0,1}}&&&\ddots & &&\\
\vdots &0&&\bigzerol &&\ddots& &\\
\vdots &\vdots && &&&\ddots &\\
0&0&&&&&&1
\end{pmatrix}.
\end{eqnarray*}
Let $H_7':\mathbb{R}^{1,n}\rightarrow \mathbb{R}^{1,n}$ 
be the linear isomorphism defined by 
\[H_7'(x)=xM^{-1}.\]
The composition of 
$H_6'\circ H_3\circ H_1\circ L_{(p_0,\ldots,p_k)}\circ H_2\circ H_4\circ H_5$ 
and $H_7'$ is 
as follows:
\begin{eqnarray*}
{}&{}&(H_6'\circ H_3\circ H_1\circ L_{(p_0,\ldots,p_k)}\circ H_2\circ H_4\circ H_5\circ H_7')(x_0,x_1,\ldots,x_n)\\
{}&=&\left(-\frac{x_1}{2\alpha _{0,1}}-\frac{\sum_{i=2}^k\alpha _{0,i}x_i}{\alpha _{0,1}},
x_2,\ldots,
x_k,x_0x_1+\sum_{i=k+1}^nx_i^2\right).
\end{eqnarray*}
%%%%%%%%%%%%%%%%%%%%%%%%%%%%%%%%%%%%%%%%%%%%%%%%%%%%%%%%%%%%%%%%%%%%%%%%%%%%
Let $H_8':\mathbb{R}^{k+1}\rightarrow \mathbb{R}^{k+1}$ 
be the linear isomorphism defined by 
\begin{eqnarray*}
H_8'(X_0,X_1,\ldots,X_k)
=\left(-2\left(\alpha _{0,1}X_0+\sum_{i=1}^{k-1}\alpha _{0,i+1}X_i\right),X_1,\ldots,X_k\right).
\end{eqnarray*}
The composition of 
$H_6'\circ H_3\circ H_1\circ L_{(p_0,\ldots,p_k)}\circ H_2\circ H_4\circ H_5\circ H_7'$ and $H_8'$ is 
as follows:
\begin{eqnarray*}
{}&{}&(H_8'\circ H_6'\circ H_3\circ H_1\circ L_{(p_0,\ldots,p_k)}\circ H_2\circ H_4\circ H_5\circ H_7')(x_0,x_1,\ldots,x_n)\\
{}&=&\left(x_1,\ldots,x_k,x_0x_1+\sum_{i=k+1}^nx_i^2\right).
\end{eqnarray*}\hfill\qed
\smallskip
%%%%%%%%%%%%%%%%%%%%%%%%%%%%%%%%%%%%%%%%%%%%%%%%%%%%%%%%%%%%%%%%%%%%%%%%%%%%%%%
%%%%%%%%%%%%%%%%%%%%%%%%%%%%%%%%%%%%%%%%%%%%%%%%%%%%%%%%%%%%%%%%%%%%%%%%%%%%%%%
\subsubsection{The case $V(p_0, \ldots, p_k)\cap T= T$}\label{special case}
%%%%%%%%%%%%%%%%%%%%%%%%%%%%%%%%%%%%%%%%%%%%%%%%%%%%%%%%%%%%%%%%%%%%%%%%%%%%%%%%%
\qquad 
The strategy of the proof in this case is the same as 
the one given in \ref{The generic case}. 
In this case also, there are four steps.   
\par      
\underline{STEP 1.}\qquad 
Step 1 is completely the same as Step 1 of \ref{The generic case}.   
Thus, by the same diffeomorphism $H_1: \mathbb{R}^{k+1}\to \mathbb{R}^{k+1}$, we have the 
following:
\begin{eqnarray*}
&{}&(H_1\circ L_{(p_0,\ldots,p_k)})(x_0,x_1,\ldots,x_n)\\
&=&
\left(
-\left(p_{10}-p_{00}\right)\left(x_0-p_{00}\right)
+\sum_{j=1}^n (p_{1j}-p_{0j})(x_j-p_{0j}),\ldots, 
%\sum_{j=0}^n (p_{k j}-p_{0j})(x_j-p_{0j}),
\right. \\ 
{ } & { } & 
\qquad  
-\left.\left(x_0-p_{00}\right)^2+\sum_{j=1}^n(x_j-p_{0j})^2\right).   
\end{eqnarray*}
\par 
\smallskip 
\underline{STEP 2.}\qquad 
By the same diffeomorphism $H_2 : \mathbb{R}^{1,n}\rightarrow \mathbb{R}^{1,n}$, 
we have the following:  
\begin{eqnarray*}
{}&{}&(H_1\circ L_{(p_0,\ldots,p_k)}\circ H_2)(x_0,x_1,\ldots,x_n)\\
{}&=&\left(\sum_{j=0}^n (p_{1j}-p_{0j})x_j,\ldots,\sum_{j=0}^n (p_{k j}-p_{0j})x_j,\, 
-x_0^2+\sum_{j=1}^nx_j^2\right)\\
{}&=&
\begin{pmatrix}
x_0,&x_1,&\cdots &x_n
\end{pmatrix}
\left(
\begin{array}{cccr}
p_{10}-p_{00}&\cdots &p_{k0}-p_{00}&-x_0\\
p_{11}-p_{01}&\cdots &p_{k1}-p_{01}&x_1\\
\vdots & & \vdots &\vdots \\
\vdots &   & \vdots &\vdots \\
p_{1n}-p_{0n}&\cdots &p_{k n}-p_{0n}&x_n
\end{array}
\right).
\end{eqnarray*}
Set 
\begin{eqnarray*}
A=\begin{pmatrix}
p_{10}-p_{00}&\cdots &p_{k0}-p_{00}\\
p_{11}-p_{01}&\cdots &p_{k1}-p_{01}\\
\vdots & & \vdots \\
\vdots &   & \vdots \\
p_{1n}-p_{0n}&\cdots &p_{k n}-p_{0n}
\end{pmatrix}.
\end{eqnarray*}
Since $(k+1)$-points $p_0,\ldots,p_k$ are in general position, 
it is clear that the rank of $(n+1)\times k$ matrix $A$ is $k$.    
Moreover, since $V(p_0, \ldots, p_k)\cap T=T$, 
there exists a $k\times k$ regular matrix $B$ such that the set of 
column vectors of $AB$ is the subset of an orthonormal 
basis of $\mathbb{R}^{n+1}$ and the matrix $AB$ has the following form:
\begin{eqnarray*}
AB=\begin{pmatrix}
1&0&\cdots &0\\
0&    \alpha _{1,2}  & \cdots &\alpha _{1,k}\\
\vdots & \vdots           & \ddots    &  \vdots  \\
0 &   \alpha _{n,2}  & \cdots &\alpha _{n,k}\\
\end{pmatrix}.
\end{eqnarray*}
\smallskip 
%%%%%%%%%%%%%%%%%%%%%%%%%%%%%%%%%%%%%%%%%%%%%%%%%%%  
\underline{STEP 3.}\qquad    
By the same diffeomorphism $H_3 : \mathbb{R}^{k+1}\rightarrow \mathbb{R}^{k+1}$ 
as in the Step 3 of the proof of $(1)$ of Theorem \ref{theorem 1}, 
we have the following:
\begin{eqnarray*}
{}&{}&\left(H_3\circ H_1\circ L_{(p_0,\ldots,p_k)}\circ H_2\right)\left(x_0,x_1,\ldots,x_n\right)\\
{}&=
&
%\begin{pmatrix}
\left( 
x_0, x_1,\ldots , x_n 
\right)
%\end{pmatrix}
\left(
\begin{array}{ccc|r}
&&&  -x_0\\ 
&&& x_1\\
&\mbox{\smash{\Large $A$}}&&\vdots \\
&&&\vdots \\
&&&x_n
\end{array}
\right)
\left(
\begin{array}{ccc|c}
&&&  0\\ 
&\mbox{\smash{\Large $B$}}&&\vdots\\
&&&0\\ \hline
0&\cdots &0&1
\end{array}
\right)\\
{}&=&
%\begin{pmatrix}
\left(x_0, x_1, \ldots , x_n\right)
%\end{pmatrix}
\left(
\begin{array}{ccccr}
 1&0&\cdots &0& -x_0\\
0&    \alpha _{1,2}  & \cdots &\alpha _{1,k}& x_1\\
\vdots & \vdots           & \ddots    &  \vdots & \vdots \\
0 &   \alpha _{n,2}  & \cdots &\alpha _{n,k} & x_{n}\\ 
\end{array}
\right) .\\ 
\end{eqnarray*}
%Let ${\bf a}_i$ be the $i$-th column vector of $AB$ $(1\leq i\leq k)$. 
%Note that $V(p_0,\ldots ,p_k)=\sum_{i=1}^k\mathbb{R}{ }^t {\bf a}_i$. 
For natural numbers $i$ $(1\leq i\leq k-1)$, 
set ${\bf \widetilde{a}}_i=(\alpha_{1,i+1},\ldots,\alpha_{n,i+1})$, 
and we see the set $\{{\bf \widetilde{a}}_1,\ldots ,{\bf \widetilde{a}}_{k-1}\}$ is a  
subset of an orthonormal basis of $\mathbb{R}^n$. 
We can choose ${\bf \widetilde{a}}_{k},\ldots ,{\bf \widetilde{a}}_n$ so that 
the set $\{{\bf \widetilde{a}}_1,\ldots ,{\bf \widetilde{a}}_n\}$ is an 
orthonormal basis of $\mathbb{R}^n$. 
Put $C=({}^t{\bf \widetilde{a}}_1,\ldots ,{}^t{\bf \widetilde{a}}_n)$. 
Note that the matrix $C$ is 
an $n\times n$ orthogonal matrix. 
By the same linear isomorphism 
%%%%%%%%%%%%%%%%%%%%%%%%%%%%%%%%%%%%%%%%%%%%%%
%Let $H_4 : \mathbb{R}^{1,n}\rightarrow \mathbb{R}^{1,n}$ be the linear isomorphism defined by 
\begin{eqnarray*}
{}&{}&H_4(x_0,x_1,\ldots,x_n)=
\left( 
x_0, x_1,\ldots , x_n 
\right)
\left(
\begin{array}{c|ccc}
1&0&\cdots &0\\
\hline 
0&&&\\
\vdots &&\mbox{\smash{\Large $C$}}&\\
0&&&  
\end{array}
\right)
\end{eqnarray*}
%The composition of $H_3\circ H_1\circ L_{(p_0,\ldots,p_k)}\circ H_2$ and $H_4$
%is as follows:
appeared in the Step 3 of the proof of $(1)$ of Theorem \ref{theorem 1}, 
we have the following:
\begin{eqnarray*}
{}&{}&(H_3\circ H_1\circ L_{(p_0,\ldots,p_k)}\circ H_2\circ H_4)(x_0,x_1,\ldots,x_n)\\
%&=&(x_0,x_1,\ldots,x_n)
%\left(
%\begin{array}{cccr}
%1  &      & \bigzerol & -x_0\\
% &                 &  1&   \vdots \\
% &     \bigzerol            &  &   \vdots \\
%  &             &   &  x_n \\
%\end{array}
%\right) \\ 
{}&=&\left(x_0,x_1,\ldots,x_{k-1},-x_0^2+\sum_{i=1}^nx_i^2\right).
\end{eqnarray*}

%%%%%%%%%%%%%%%%%%%%%%%%%%%%%%%%%%%%%%%%%%%%%%%%%%%%%%%%%%
\underline{STEP 4.}\qquad 
This is the last step. Firstly, in order to remove $x_0^2,\ldots,x_{k-1}^2$ of the last 
component, we construct the diffeomorphism of the target space 
$\widetilde{H}_5 : \mathbb{R}^{k+1}\rightarrow \mathbb{R}^{k+1}$ below:
\begin{eqnarray*}
{}&{}&\widetilde{H}_5(X_0,X_1,\ldots,X_k)=
\left( 
X_0, X_1,\ldots , X_{k-1}, X_k+X_0^2-\sum_{i=1}^{k-1}X_i^2 
\right).
\end{eqnarray*}
The composition of $H_3\circ H_1\circ L_{(p_0,\ldots,p_k)}\circ H_2\circ H_4$ and $\widetilde{H}_5$
is as follows:
\begin{eqnarray*}
{}&{}&(\widetilde{H}_5\circ H_3\circ H_1\circ L_{(p_0,\ldots,p_k)}\circ H_2\circ H_4)(x_0,x_1,\ldots,x_n)\\
&=&\left(x_0,x_1,\ldots,x_{k-1},\sum_{i=k}^nx_i^2\right).
\end{eqnarray*}
%%%%%%%%%%%%%%%%%%%%%%%%%%%%%%%%%%%%%%%%%%%%%%%%%%%%%%%%%%%%%%%%%%%
Let  $\widetilde{H}_6 : \mathbb{R}^{1,n}\rightarrow \mathbb{R}^{1,n}$ be the linear isomorphism defined by 
\begin{eqnarray*}
{}&{}&\widetilde{H}_6(x_0,x_1,\ldots,x_n)=
\left( 
x_1,x_2,\ldots , x_k,x_0,x_{k+1},\ldots,x_n
\right).
\end{eqnarray*}
The composition of $H_3\circ H_1\circ L_{(p_0,\ldots,p_k)}\circ H_2\circ H_4\circ \widetilde{H}_5$ 
and $\widetilde{H}_6$
is as follows:
\begin{eqnarray*}
{}&{}&(\widetilde{H}_5\circ H_3\circ H_1\circ L_{(p_0,\ldots,p_k)}\circ H_2\circ H_4\circ \widetilde{H}_6)
(x_0,x_1,\ldots,x_n)\\
&=&\left(x_1,\ldots,x_k,x_0^2+\sum_{i=k+1}^nx_i^2\right).
\end{eqnarray*}
\hfill\qed
%%%%%%%%%%%%%%%%%%%%%%%%%%%%%%%%%%%%%%%%%%%%%%%%%%%%%%%%%%%%%%%%%%
%%%%%%%%%%%%%%%%%%%%%%%%%%%%%%%%%%%%%%%%%%%%%%%%%%%%%%%%%%%%%%%%%
\begin{remark}\label{remark 3.1}
%%%%%%%%%%%%%%%%%%%%%%%%%%%%%%%%%%%%%%%%%%%%%%%%%%%%%%%%%%%%%%%%%%
{\rm 
\begin{enumerate}
\item Let $f_{{p_0p_i}}: \mathbb{R}^{1,n}\to \mathbb{R}$ be the linear function 
defined by $f_{{p_0 p_i}}(x)=\overrightarrow{p_0p_i}\cdot (x-p_i)$ for any $i$ $(1\le i\le n-1)$.      
Then, Step 1 and Step 2 of the proof of $(1)$ of Theorem \ref{theorem 1} 
imply the following for any 
$y\in \mathbb{R}^{n}$.       
\[ 
L_{(p_0, \ldots, p_{n-1})}^{-1}(y)=
H_2\left(\left(f_{{p_0 p_1}}, \ldots, f_{{p_0 p_{n-1}}}, \ell_{p_0}^2\right)^{-1}(H_1(y))\right).   
\]
\item Since the proof of $(1)$ of Theorem \ref{theorem 1} 
requires only affine diffeomorphisms for diffeomorphisms of the source 
space, we can choose an affine diffeomorphism 
as the diffeomorphism 
$h: \mathbb{R}^{1,n}\to \mathbb{R}^{1,n}$ in Corollary \ref{corollary 1}.   
\end{enumerate}
}
\end{remark}
%%%%%%%%%%%%%%%%%%%%%%%%%%%%%%%%%%%%%%%%%%%%%%%%%% 
%%%%%%%%%%%%%%%%%%%%%%%%%%%%%%%%%%%%%%%%%%%%%%%%%%
\subsection{Proof of (2) of Theorem \ref{theorem 1}}
%%%%%%%%%%%%%%%%%%%%%%%%%%%%%%%%%%%%%%%%%%%%%%%%%% 
The strategy of the proof of $(2)$ of Theorem 1 
is the same as the strategy of the proof of $(1)$ of Theorem 1.    
In the case that $V(p_0,\ldots,p_n)$ is space-like, 
compose the mapping 
$H_6\circ H_3\circ H_1\circ L_{(p_0,\ldots,p_n)}\circ H_2\circ H_4\circ H_5\circ H_7$ 
and the linear coordinate transformation of the target 
$(X_0,X_1,\ldots,X_n)\mapsto (X_0,X_1,\ldots,-X_n)$.    
\hfill\qed 
\par 
\smallskip 
%%%%%%%%%%%%%%%%%%%%%%%%%%%%%%%%%%%%%%%%%%%%%%%%%%
%%%%%%%%%%%%%%%%%%%%%%%%%%%%%%%%%%%%%%%%%%%%%%%%%%
\subsection{Proof of (3) of Theorem \ref{theorem 1}}
%%%%%%%%%%%%%%%%%%%%%%%%%%%%%%%%%%%%%%%%%%%%%%%%%%%
The strategy of the proof of $(3)$ of Theorem 1 
is the same as the strategy of the proof of $(1)$ of Theorem 1.    
In this case, since the rank of the $(n+1)\times k$ matrix $A$ 
appeared in Step 2 of the proof of 
$(1)$ of Theorem \ref{theorem 1} is $n+1$, there exists a $k\times k$ regular 
matrix $B$ such that the following holds:
 \begin{eqnarray*}
A{B}=
\begin{pmatrix}
1&&\mbox{\smash{\Large $0$}}&         0 & \cdots & 0   \\
&\ddots &&        \vdots &   & \vdots \\
\mbox{\smash{\Large $0$}}&&1&            0  &  \cdots  &  0\\
\end{pmatrix}. 
\end{eqnarray*}
By the same method as the proof of $(1)$ of Theorem \ref{theorem 1}, we have 
\begin{eqnarray*}
{}&{}&\left(H_3\circ H_1\circ L_{(p_0,\ldots,p_k)}\circ H_2\right)\left(x_0,x_1,\ldots,x_n\right)\\
{}&=&
\left( 
x_0, x_1,\ldots , x_n 
\right)
\left(
\begin{array}{ccc|r}
&&&  -x_0\\ 
&&& x_1\\
&\mbox{\smash{\Large $A$}}&&\vdots \\
&&&\vdots \\
&&&x_n
\end{array}
\right)
\left(
\begin{array}{ccc|c}
&&&  0\\ 
&\mbox{\smash{\Large $B$}}&&\vdots\\
&&&0\\ \hline
0&\cdots &0&1
\end{array}
\right)\\
{}&=&
%\begin{pmatrix}
\left(x_0, x_1, \ldots , x_n\right)
%\end{pmatrix}
\begin{pmatrix}
1&&\mbox{\smash{\Large $0$}}&    0 & \cdots & 0 &-x_0  \\
&\ddots &&        \vdots &   & \vdots &\vdots  \\
\mbox{\smash{\Large $0$}}&&1&            0  &  \cdots  &  0 &x_n\\
\end{pmatrix}\\
&=&\left(x_0,x_1,\ldots ,x_n,0,\ldots ,0,-x_0^2+\sum_{i=1}^nx_i^2\right).\\
\end{eqnarray*}
Finally, by composing $H_3\circ H_1\circ L_{(p_0,\ldots,p_k)}\circ H_2$ and 
$\widetilde{H}_5$ appeared in the proof of $(1)$ of Theorem \ref{theorem 1}, we have the following: 
%Let  $\widetilde{H}_4 : \mathbb{R}^{k+1}\rightarrow \mathbb{R}^{k+1}$ 
%be the diffeomorphism defined by 
%\begin{eqnarray*}
%{}&{}&\widetilde{H}_4(X_0,X_1,\ldots,X_k)=\left(X_0,X_1,\ldots,X_{k-1},X_k+X_0^2
%-\sum_{i=1}^nX_i^2\right).
%\end{eqnarray*}
%The composition of $H_3\circ H_1\circ L_{(p_0,\ldots,p_k)}\circ H_2$ and 
%$\widetilde{H}_4$ is as follows:
\begin{eqnarray*}
{}&{}&(\widetilde{H}_5\circ H_3\circ H_1\circ L_{(p_0,\ldots,p_k)}\circ H_2)
(x_0,x_1,\ldots,x_n)=\left(x_0,x_1,\ldots,x_n,0,\ldots,0\right).
\end{eqnarray*}
\hfill\qed
%%%%%%%%%%%%%%%%%%%%%%%%%%%%%%%%%%%%%%%%%%%%%%%%%%%%%%%%%%%%%%
%%%%%%%%%%%%%%%%%%%%%%%%%%%%%%%%%%%%%%%%%%%%%%%%%%%%%%%%%%%%%%%%
\section{Appendix}
%%%%%%%%%%%%%%%%%%%%%%%%%%%%%%%%%%%%%%%%%%%%%%%%%%%%%%%%
By the similar proof as in Section \ref{section 3}, we can 
obtain a recognizable classification of 
Lorentzian distance-squared mappings 
when the given $(k+1)$-points are not in general position as follows. 
%%%%%%%%%%%%%%%%%%%%%%%%%%%%%%%%%%%%%%%%%%%%%%%%%%
\begin{theorem}\label{not in general position}
%%%%%%%%%%%%%%%%%%%%%%%%%%%%%%%%%%%%%%%%%%%%%%%%%%%%
Let $j, k$ be two positive integers satisfying $j<k$ and 
let $\tau : \mathbb{R}^{j+1}\rightarrow \mathbb{R}^{k+1}$ be the 
inclusion:
\[\tau (X_0,X_1,\ldots,X_j)=(X_0,X_1,\ldots,X_j,0,\ldots,0).\]
\begin{enumerate}
\item 
Let $j, k, n$ be three positive integers satisfying $j<n, j<k$, 
and let $p_0, \ldots, p_k$$\in \mathbb{R}^{1,n}$ be $(k+1)$ 
points such that $\dim V(p_0, \ldots, p_j)=\dim V(p_0, \ldots, p_k)=j$. 
Then, the following hold:   
\begin{enumerate}
\item The Lorentzian distance-squared mapping 
$L_{(p_0, \ldots, p_k)}: \mathbb{R}^{1,n}\to \mathbb{R}^{k+1}$ is 
$\mathcal{A}$-equivalent to the mapping $\tau \circ \Phi_j$ if and only if 
the recognition subspace $V(p_0, \ldots, p_k)$ is time-like. 
\item The Lorentzian distance-squared mapping 
$L_{(p_0, \ldots, p_k)}:\mathbb{R}^{1,n}\to \mathbb{R}^{k+1}$ is 
$\mathcal{A}$-equivalent to the mapping $\tau \circ \Psi_j$ if and only if 
the recognition subspace 
$V(p_0, \ldots, p_k)$ is space-like. 
\item The Lorentzian distance-squared mapping 
$L_{(p_0, \ldots, p_k)}:\mathbb{R}^{1,n}\to \mathbb{R}^{k+1}$ is 
$\mathcal{A}$-equivalent to the mapping 
\[
(x_0, \ldots, x_n)\mapsto \left( x_1, \ldots, x_j, x_0x_1
+\sum_{i=j+1}^n x_i^2,0,\ldots,0\right)
\]
if and only if 
%the intersection $V(p_0, \ldots, p_k)\cap LC$ is a line.  
the recognition subspace 
$V(p_0, \ldots, p_k)$ is light-like. 
\end{enumerate}
\item 
Let $k, n$ be two positive integers satisfying $n<k$ and 
let $p_0, \ldots, p_k$$\in \mathbb{R}^{1,n}$ be $(k+1)$ 
points such that $\dim V(p_0,\ldots,p_n)=\dim V(p_0,\ldots,p_k)=n$
Then, the following hold: 
\begin{enumerate}
\item The Lorentzian distance-squared mapping 
$L_{(p_0, \ldots, p_k)}:\mathbb{R}^{1,n}\to \mathbb{R}^{k+1}$ is 
$\mathcal{A}$-equivalent to the normal form of 
definite fold mapping $\tau \circ \Phi_n$ if and only if 
$V(p_0, \ldots, p_k)$ is time-like or space-like. 
\item The Lorentzian distance-squared mapping 
$L_{(p_0, \ldots, p_k)}: \mathbb{R}^{1,n}\to \mathbb{R}^{k+1}$ is 
$\mathcal{A}$-equivalent to the mapping  
$(x_0, \ldots, x_n)\mapsto (x_1, \ldots, x_n, x_0x_1,0\ldots,0)$ 
if and only if 
the recognition subspace 
$V(p_0, \ldots, p_k)$ is light-like.
\end{enumerate}
\item Let $k,n$ be two positive integers and let $p_0,\ldots,p_k\in \mathbb{R}^{n,1}$ 
be the same point. Then, the Lorentzian distance-squared mapping 
$L_{(p_0,\ldots,p_k)}:\mathbb{R}^{n,1}\rightarrow \mathbb{R}^{k+1}$ is 
$\mathcal{A}$-equivalent to the mapping 
\[(x_0,\ldots,x_n)\mapsto \left(-x_0^2+\sum_{i=1}^nx_i^2,0,\ldots,0\right).\]
\end{enumerate}
\end{theorem}
%%%%%%%%%%%%%%%%%%%%%%%%%%%%%%%%%%%%%%%%%%%%
\begin{proof}
%%%%%%%%%%%%%%%%%%%%%%%%%%%%%%%%%%%%%%%% 
Firstly, the generic case of $(1)$ of Theorem \ref{not in general position} is proved. 
The strategy of the proof in this case is the same as the proof 
of the generic case of $(1)$ of Theorem \ref{theorem 1}. We have the following: 
\begin{eqnarray*}
{}&{}&(H_1\circ L_{(p_0,\ldots,p_k)}\circ H_2)(x_0,x_1,\ldots,x_n)\\
{}&=&
\begin{pmatrix}
x_0,&x_1,&\cdots &x_n
\end{pmatrix}
\left(
\begin{array}{cccr}
p_{10}-p_{00}&\cdots &p_{k0}-p_{00}&-x_0\\
p_{11}-p_{01}&\cdots &p_{k1}-p_{01}&x_1\\
\vdots & & \vdots &\vdots \\
\vdots &   & \vdots &\vdots \\
p_{1n}-p_{0n}&\cdots &p_{kn}-p_{0n}&x_n
\end{array}
\right).
\end{eqnarray*}
Set 
\begin{eqnarray*}
A=\begin{pmatrix}
p_{10}-p_{00}&\cdots &p_{j0}-p_{00}\\
p_{11}-p_{01}&\cdots &p_{j1}-p_{01}\\
\vdots & & \vdots \\
\vdots &   & \vdots \\
p_{1n}-p_{0n}&\cdots &p_{jn}-p_{0n}
\end{pmatrix}
= 
\left(
\begin{array}{ccc}
p_{10}-p_{00}&\cdots &p_{j0}-p_{00}  \\
 \hline 
{ } & { } & { }  \\
{ } & { } & { } \\
{ } & \mbox{\smash{\LARGE $\widetilde{A}$}}& { } \\
{ } & { } & { } \\ 
{ } & { } & { } \\ 
\end{array}
\right).
\end{eqnarray*}
Since $(j+1)$-points $p_0,\ldots,p_j$ are in general position, 
it is clear that the rank of $(n+1)\times j$ matrix $A$ is $j$.    
Moreover, since $V(p_0, \ldots, p_j)\cap T=\{0\}$,  
the $n\times j$ matrix $\widetilde{A}$ has the same rank $j$.
There exists a $j\times j$ regular matrix $B$ such that the set of 
column vectors of $\widetilde {A}B$ is the subset of an orthonormal 
basis of $\mathbb{R}^{n}$. Set
\begin{eqnarray*}
\widetilde {A}B=\begin{pmatrix}
    \alpha _{1,1}  & \cdots &\alpha _{1,j}\\
 \vdots           & \ddots    &  \vdots  \\
   \alpha _{n,1}  & \cdots &\alpha _{n,j}\\
\end{pmatrix}.
\end{eqnarray*}
Let $H_3 : \mathbb{R}^{k+1}\rightarrow \mathbb{R}^{k+1}$ be 
the linear isomorphism defined by 
\begin{eqnarray*}
{}&{}&H_3(X_0,X_1,\ldots,X_k )=
\left( X_0, X_1, \ldots , X_k\right)
\left(
\begin{array}{ccc|ccc}
&&& & \\ 
&\mbox{\smash{\Large $B$}}&&&\bigzerol \\
&&&& &\\ \hline
&&&&&\\ 
&\bigzerol &&& \mbox{\smash{\Large $E_{k-j+1}$}}&\\ 
&&& &&
\end{array}
\right), 
\end{eqnarray*}
where $E_{k-j+1}$ is a $(k-j+1)\times (k-j+1)$ unit matrix. 
Thus, we obtain 
\begin{eqnarray*}
{}&{}&\left(H_3\circ H_1\circ L_{(p_0,\ldots,p_k)}\circ H_2\right)
\left(x_0,x_1,\ldots,x_n\right)\\
{}&=&
%\begin{pmatrix}
\left(x_0, x_1, \ldots , x_n\right)
%\end{pmatrix}
\left(
\begin{array}{ccccccc}
\alpha _{0,1}&\cdots &\alpha _{0,j}&p_{j+1,0}-p_{00}&\cdots  &p_{k,0}-p_{00}&-x_0\\
\alpha _{1,1}&\cdots  &\alpha _{1,j}&p_{j+1,1}-p_{01}&\cdots  &p_{k,1}-p_{01}& x_1\\
\vdots             & \ddots &  \vdots &\vdots & &\vdots & \vdots  \\
\alpha _{n,1}&\cdots &\alpha _{n.j} &p_{j+1,n}-p_{0n}&\cdots  &p_{k,n}-p_{0n}& x_{n}\\
\end{array}
\right). 
\end{eqnarray*}
Set ${\bf a}_i=(\alpha_{0,i},\alpha_{1,i},\ldots,\alpha_{n,i})$ $(1\leq i\leq j)$. 
Since $V(p_0,\ldots,p_k)=\sum_{i=1}^j\mathbb{R}{\bf a}_i$, 
there exists a $k\times k$ regular matrix $\widetilde{B}$ such that 
\begin{eqnarray*}
\left(
\begin{array}{ccccccc}
\alpha _{0,1}&\cdots &\alpha _{0,j}&p_{j+1,0}-p_{00}&\cdots  &p_{k,0}-p_{00}\\
\alpha _{1,1}&\cdots  &\alpha _{1,j}&p_{j+1,1}-p_{01}&\cdots  &p_{k,1}-p_{01} \\
\vdots             & \ddots &  \vdots &\vdots & &\vdots   \\
\alpha _{n,1}&\cdots &\alpha _{n.j} &p_{j+1,n}-p_{0n}&\cdots  &p_{k,n}-p_{0n}\\
\end{array}
\right)\widetilde{B} \\
=
\left(
\begin{array}{ccccccc}
\alpha _{0,1}&\cdots &\alpha _{0,j}&0&\cdots  &0\\
\alpha _{1,1}&\cdots  &\alpha _{1,j}&\vdots & &\vdots \\
\vdots             & \ddots &  \vdots &\vdots & & \vdots  \\
\alpha _{n,1}&\cdots &\alpha _{n.j} &0&\cdots  &0\\
\end{array}\right).
\end{eqnarray*}
By composing $H_3\circ H_1\circ L_{(p_0,\ldots,p_k)}\circ H_2$ and 
the linear isomorphism of the target defined by 
\begin{eqnarray*}
(X_0,X_1,\ldots,X_k)\mapsto 
\left( X_0, X_1, \ldots , X_k\right)
\left(
\begin{array}{ccc|c}
&&&  0\\ 
&\mbox{\smash{\Large $\widetilde{B}$}}&&\vdots\\
&&&0\\ \hline
0&\cdots &0&1
\end{array}
\right),
\end{eqnarray*}
we may assume that
\begin{eqnarray*}
{}&{}&\left(H_3\circ H_1\circ L_{(p_0,\ldots,p_k)}\circ H_2\right)
\left(x_0,x_1,\ldots,x_n\right)\\
{}&=&
%\begin{pmatrix}
\left(x_0, x_1, \ldots , x_n\right)
\left(
\begin{array}{ccccccc}
\alpha _{0,1}&\cdots &\alpha _{0,j}&0&\cdots  &0&-x_0\\
\alpha _{1,1}&\cdots  &\alpha _{1,j}&\vdots & &\vdots &x_1 \\
\vdots             & \ddots &  \vdots &\vdots & & \vdots &\vdots \\
\alpha _{n,1}&\cdots &\alpha _{n.j} &0&\cdots  &0&x_0 \\
\end{array}\right).
\end{eqnarray*}
Set ${\bf \widetilde{a}}_i=(\alpha_{1,i},\ldots,\alpha_{n,i})$ $(1\leq i\leq j)$. 
We can choose ${\bf \widetilde{a}}_{j+1},\ldots ,{\bf \widetilde{a}}_n$ so that 
the set $\{{\bf \widetilde{a}}_1,\ldots ,{\bf \widetilde{a}}_n\}$ is an 
orthonormal basis of $\mathbb{R}^n$. 
Put $C=({}^t{\bf \widetilde{a}}_1,\ldots ,{}^t{\bf \widetilde{a}}_n)$. 
By $H_4$ appeared in the proof of $(1)$ of Theorem \ref{theorem 1}, 
we obtain 
\begin{eqnarray*}
{}&{}&(H_3\circ H_1\circ L_{(p_0,\ldots,p_k)}\circ H_2\circ H_4)(x_0,x_1,\ldots,x_n)\\
%&=&(x_0,x_1,\ldots,x_n)
%\left(
%\begin{array}{cccr}
%\alpha _{0,1}  & \cdots &\alpha _{0,k}   & -x_0\\
%1  &      & \bigzerol & x_1\\
% &       \ddots        &  &   \vdots \\
% &                 &  1&   \vdots \\
% &     \bigzerol            &  &   \vdots \\
%  &             &   &  x_n \\
%\end{array}
%\right) \\ 
{}&=&\left(\alpha_{0,1}x_0+x_1,\ldots,\alpha _{0,j}x_0+x_j, 0,\ldots,0,-x_0^2
+\sum_{i=1}^nx_i^2\right).
\end{eqnarray*}
The rest of this proof is the same as the strategy of $(1)$ of 
Theorem \ref{theorem 1}.

Next, the case that $V(p_0,\ldots,p_k)\cap T=T$ is proved. 
For the $(n+1)\times j$ matrix $A$ appeared in the proof of the generic case 
$(1)$ of Theorem \ref{not in general position}, 
since the rank of $A$ is $j$ and $V(p_0,\ldots,p_j)\cap T=T$, 
there exists a $j\times j$ regular matrix $B$ such that 
the set of column vectors of $AB$ is the subset 
of an orthonormal basis of $\mathbb{R}^{n+1}$ and 
the matrix $AB$ has the following form:
\begin{eqnarray*}
AB=\begin{pmatrix}
1&0&\cdots &0\\
0&    \alpha _{1,2}  & \cdots &\alpha _{1,j}\\
\vdots & \vdots           & \ddots    &  \vdots  \\
0 &   \alpha _{n,2}  & \cdots &\alpha _{n,j}\\
\end{pmatrix}.
\end{eqnarray*}
The rest of this proof is clear by seeing the proof of 
the generic case of $(1)$ of Theorem \ref{not in general position} and 
the proof of \ref{special case} of $(1)$ of Theorem \ref{theorem 1}.

%%%%%%%%%%%%%%%%%%%%%%%%%%%%%%%%%%%%%%%%%%%%%%%%%%%%%%%
The proof of $(2)$ of Theorem \ref{not in general position} 
is the same as the proof of $(1)$ of Theorem \ref{not in general position}. 
%%%%%%%%%%%%%%%%%%%%%%%%%%%%%%%%%%%%%%%%%%%%%%%%%%%%%%%

Finally, the case $(3)$ is proved.    
By composing 
$L_{(p_0,\ldots,p_k)}$, 
the linear isomorphism 
of the target defined by 
$(X_0,X_1,\ldots,X_k)\mapsto (X_0,X_1-X_0,\ldots,X_k-X_0)$ 
and 
%Next, compose its composition and 
the linear isomorphism 
of the source defined by 
$(x_0,x_1,\ldots,x_n)\mapsto (x_0+p_{00},x_1+p_{01},\ldots,x_n+p_{0n})$, 
the desired mapping is obtained.
\end{proof}
\section*{Acknowledgements}
The authors thank G. Ishikawa, S. Izumiya, T. Ohmoto, M. A. S. Ruas, 
O. Saeki, K. Saji and F. Tari for helpful correspondences. 
%%%%%%%%%%%%%%%%%%%%%%%%%%%%%%%%%%%%%%%%%%%%%%%%%%%%%%%%%%%%%
%%%%%%%%%%%%%%%%%%%%%%%%%%%%%%%%%%%%%%%%%%%%%%%%%%%%%%%%%%%%%
%\subsection{Proof of Lemma \ref{lemma1}}
%%%%%%%%%%%%%%%%%%%%%%%%%%%%%%%%%%%%%%%%%%%%%%%%%%%%%%%%%%%%%

%\par 
%\smallskip 
%%%%%%%%%%%%%%%%%%%%%%%%%%%%%%%%%%%%%%%%%%%%%%%%%%   
%%%%%%%%%%%%%%%%%%%%%%%%%%%%%%%%%%%%%%%%%%%%%%%%%%
%\section*{References}

\end{document}